\documentclass[11pt]{article}
\usepackage[a4paper,margin=1in]{geometry}
\usepackage{amsmath,amssymb,amsthm,mathtools}
\usepackage{enumitem}
\usepackage{graphicx}
\usepackage{booktabs}
\usepackage[T1]{fontenc}
\usepackage{lmodern}
\usepackage{microtype}
\usepackage[colorlinks=true,linkcolor=blue,citecolor=blue,urlcolor=blue]{hyperref}
\usepackage{authblk}

\newtheorem{theorem}{Theorem}
\newtheorem{lemma}[theorem]{Lemma}
\newtheorem{proposition}[theorem]{Proposition}
\newtheorem{remark}[theorem]{Remark}

\newcommand{\Z}{\mathbb Z}
\newcommand{\Q}{\mathbb Q}
\newcommand{\Ztwo}{\mathbb Z_2}
\newcommand{\Qtwo}{\mathbb Q_2}
\newcommand{\ordtwo}{\nu_2}

\title{The 2-Adic Valuation of the Order of the All-Ones Class in the Sandpile Group of a Square}
\author[1,6]{Turgay Akyar}

\author[1]{Artem Beliakov}

\author[1]{Konstantin Delchev}

\author[2,3]{Nikita Kalinin}

\author[1]{Ernesto Lupercio}

\author[1]{Higinio Serrano}

\author[1]{Mikhail Shkolnikov}

\author[4]{Daniel Tabares}

\author[5]{Nikolai Terekhov}

\affil[1]{Institute of Mathematics and Informatics, Bulgarian Academy of Sciences}

\affil[2]{Guangdong Technion Israel Institute of Technology (GTIIT),
241 Daxue Road, Shantou, Guangdong Province 515603, P.R. China, nikita.kalinin@gtiit.edu.cn}

\affil[3]{Technion-Israel Institute of Technology, Haifa, 32000, Haifa District, Israel, nikita.kalin@technion.ac.il}

\affil[4]{Department of Mathematics, Center for Research and Advanced Studies of the National Polytechnic Institute (CINVESTAV), Unidad Zacatenco, Mexico City, Mexico, danieles-960@hotmail.com}

\affil[5]{Laboratory of Combinatorial and Geometric Structures, Moscow Institute of Physics and Technology, Institutsky lane 9, Dolgoprudny, Moscow region, 141700, Russia, nikolayterek@gmail.com}

\affil[6]{Department of Mathematics, Middle East Technical University, 06800 Ankara, Türkiye}

\begin{document}
\maketitle

\begin{abstract}
Place one grain at every nonsink vertex of the wired $n\times n$ square,
and let $L(n)$ be the order of this operation in the sandpile group.  Thus
$L(n)$ is the least positive $q$ for which $q$ uniform grain layers form an
integral combination of toppling moves.  We prove that, for every $n\ge1$,
\[
\nu_2(L(n))=
\begin{cases}
2,&n=1,\\
1,&n\ge2\text{ even},\\
\nu_2(n+1)+2,&n\ge3\text{ odd}.
\end{cases}
\]
For even squares, this follows from the domino--sandpile results of
Florescu, Morar, Perkinson, Salter, and Xu, completed by a short parity
observation.  For odd squares, a unimodular cyclic basis identifies the
folded cokernel with a quotient by two shifted Chebyshev polynomials and
sends the all-ones class to $1$.  Its order is determined by the constant
part of this polynomial ideal, not just by a determinant.  Two normalized
Euclidean remainders reduce to consecutive Fibonacci polynomials over
$\mathbb F_2$, giving the exact valuation.
\end{abstract}

\section{Introduction: 2-adic phenomena in domino tiling}

Exact powers of $2$ occur throughout planar combinatorics and arithmetic,
often for structural rather than accidental reasons.
Classical Pfaffian formulas count domino tilings by cyclotomic trigonometric
products \cite{Kasteleyn1961,TemperleyFisher1961}.  For the ordinary
$2r\times2r$ square board $R_{2r,2r}$, the number of tilings satisfies
\[
  \mathsf M(R_{2r,2r})=2^r a_r^2,\qquad a_r\ \text{odd},
\]
where we choose $a_r>0$.  This square-board phenomenon was recorded in the
problem of John, Sachs, and Zernitz \cite{JohnSachsZernitz1987}; Pachter
gave a combinatorial proof of its exact $2$-divisibility
\cite{Pachter1997}.  John and Sachs studied the residual sequence $a_r$
modulo powers of $2$, including explicit congruences modulo $64$
\cite{JohnSachs2000}.  Cohn proved that this sequence extends continuously
to $\Ztwo$ \cite{Cohn1999}.  At the extreme, the Aztec
diamond has $\mathsf M(\operatorname{AD}_r)=2^{r(r+1)/2}$
\cite{EKLP1992I,EKLP1992II}.

Jockusch explained square-or-twice-square matching counts for suitable
graphs with rotational symmetry \cite{Jockusch1994}; Ciucu's reflection
factorization extracts powers of two from matching counts
\cite{Ciucu1997}.  Mod-$2$ adjacency kernels provide another way to control
matching divisibility
\cite{BarkleyLiu2021}.  The Temperley correspondence links planar matchings
to spanning trees \cite{KPW2000}, and Kirchhoff's theorem links spanning
trees to sandpile-group orders.  Chebyshev polynomials already enter this
circle through graph determinants: Boesch and Prodinger used them to
derive spanning-tree formulas \cite{BoeschProdinger1986}, and Strehl
expressed rectangular domino counts using resultants \cite{Strehl2001}.

We were inspired most directly by the article of Florescu, Morar,
Perkinson, Salter, and Xu \cite{FlorescuEtAl2015}.  They relate the number
of reflection-symmetric recurrent grid sandpiles to weighted domino counts
and products of Chebyshev values.  More directly, their Section~5 asks for
the order of the all-ones class, relates it to the all-twos class, and
describes it in terms of denominators of a rational toppling script
\cite[Proposition~25]{FlorescuEtAl2015}.  For an even square, their
Corollary~30 says that the all-twos order divides the odd factor $a_r$.
With one additional parity observation this gives exact $2$-primary order
$2$ for the all-ones class.  Inspired by this article, we make that
consequence explicit and develop the question into the square-grid theorem. 

For a nonsingular integral $N\times N$ matrix $A$, write
$\operatorname{coker}A=\Z^N/A\Z^N$.  It is useful to distinguish three
levels of information:
\[
 \nu_2(\det A)
 \quad\longrightarrow\quad
 \operatorname{SNF}_{(2)}(A)
 \quad\longrightarrow\quad
 \nu_2\bigl(\operatorname{ord}([b]\in\operatorname{coker}A)\bigr).
\]
The determinant records a total power of two, the local Smith form records
its distribution, and the last invariant asks where one chosen class $[b]$ lies.
The arrows indicate increasingly refined questions, not that the preceding
invariant determines the next one.  Kuperberg's Kasteleyn cokernels place
the passage from matching determinants to integral matrix equivalence and
Smith forms in a general framework \cite{Kuperberg2002}.  Arithmetic of
individual graph-Jacobian classes also has an established literature:
Lorenzini studies orders of vertex-difference classes, and Shokrieh studies
the monodromy pairing and discrete logarithms
\cite{Lorenzini2000,Shokrieh2010}.

Our focus is the exact $2$-part of the order of the canonical all-ones
class.  For an odd square, the key integral step is a unimodular cyclic
basis for the folded one-dimensional operator.  It identifies the folded
cokernel with a quotient by two shifted Chebyshev polynomials and carries
the chosen vector to the constant polynomial $1$
(Proposition~\ref{prop:cyclic-cokernel}).  We can therefore compute its
order by intersecting the polynomial ideal with $\Ztwo$, a finer operation
than evaluating a resultant.  The specific calculation is unusually
short: after removing two exact powers of $2$, the Euclidean remainders
become consecutive Fibonacci polynomials modulo $2$.  To our knowledge,
the odd-square formula below is not established in the cited literature;
we make no priority claim for the general cokernel or cyclic-basis methods.

\section{The sandpile problem and theorem}

Take the $n\times n$ square of nonsink vertices and wire its boundary to one
sink.  A toppling at a vertex sends one grain in each of the four coordinate
directions; a grain crossing the boundary disappears into the sink.  Two
integer configurations are toppling-equivalent when their difference is an
integral combination of these toppling moves.

The class $\mathbf1$ means adding one grain simultaneously at every nonsink
vertex and then stabilizing.  Its order $L(n)$ is the period of this
operation on recurrent sandpiles: it is the least $q>0$ for which $q$
uniform grain layers are toppling-equivalent to zero.  The theorem below
determines exactly the power of $2$ dividing this period.  Here
$\nu_2(x)$ denotes the exponent of $2$ in a nonzero rational number $x$,
while $\Ztwo$ and $\Qtwo$ denote the $2$-adic integers and numbers.

Let $A_n$ be the reduced Laplacian of the wired $n\times n$ square,
\[
A_n=D_n\otimes I_n+I_n\otimes D_n,
\qquad
D_n=
\begin{pmatrix}
2&-1\\
-1&2&-1\\
&\ddots&\ddots&\ddots\\
&&-1&2&-1\\
&&&-1&2
\end{pmatrix}.
\]
If $z=(z_1,\ldots,z_n)^T$ and $z_0=z_{n+1}=0$, then
\[
 z^TD_nz=\sum_{i=0}^{n}(z_{i+1}-z_i)^2.
\]
Thus $D_n$, and hence its Kronecker sum $A_n$, is positive definite, so
$A_n^{-1}$ exists over $\Q$.
Let $\mathbf 1\in\Z^{n^2}$ be the all-ones vector.  Algebraically, the
sandpile group is
\[
K_n=\Z^{n^2}/A_n\Z^{n^2}.
\]
The operation above represents the class of $\mathbf1$, and $L(n)$ is its
order.  Equivalently, $L(n)$ is the least positive integer $q$ for which
\[
qA_n^{-1}\mathbf 1\in\Z^{n^2}.
\]
Thus $A_n^{-1}\mathbf1$ is the unique rational toppling script for one
uniform layer, and $L(n)$ is the least integer that clears all of its
coordinate denominators.
Taking the largest power of $2$ occurring among the coordinate
denominators gives the exact local reformulation
\[
\ordtwo L(n)=\min\{r\ge0:2^rA_n^{-1}\mathbf1\in\Ztwo^{n^2}\}.
\]
Indeed, if the reduced coordinate denominators are $b_i$, then
$L(n)=\operatorname{lcm}_i b_i$, and both sides of the displayed equality
are $\max_i\nu_2(b_i)$.

\begin{theorem}[Square-grid valuation]\label{thm:main}
For every $n\ge1$,
\[
\boxed{
\ordtwo L(n)=
\begin{cases}
2,&n=1,\\
1,&n\ge2\text{ even},\\
\ordtwo(n+1)+2,&n\ge3\text{ odd}.
\end{cases}}
\]
\end{theorem}

\subsubsection{Example}
The valuation $\nu_2(L(n))$ is the maximum of a much finer entrywise
profile.  Reshape
$A_n^{-1}\mathbf1$ as $W^{(n)}=(w^{(n)}_{ij})$ and define the $2$-adic
denominator depth of an entry by
\[
d^{(n)}_{ij}
=\nu_2\!\left(\operatorname{den}(w^{(n)}_{ij})\right)
=\max\{0,-\nu_2(w^{(n)}_{ij})\}.
\]
Then $\max_{i,j}d^{(n)}_{ij}=\nu_2(L(n))$.  Figure~\ref{fig:n31-depth}
shows the complete profile at $n=31$.  A displayed zero means that the
coordinate has odd denominator; it does not assert that the coordinate or
its ordinary $2$-adic valuation is zero.

\begin{figure}[tbp]
\centering
\includegraphics[width=\linewidth]{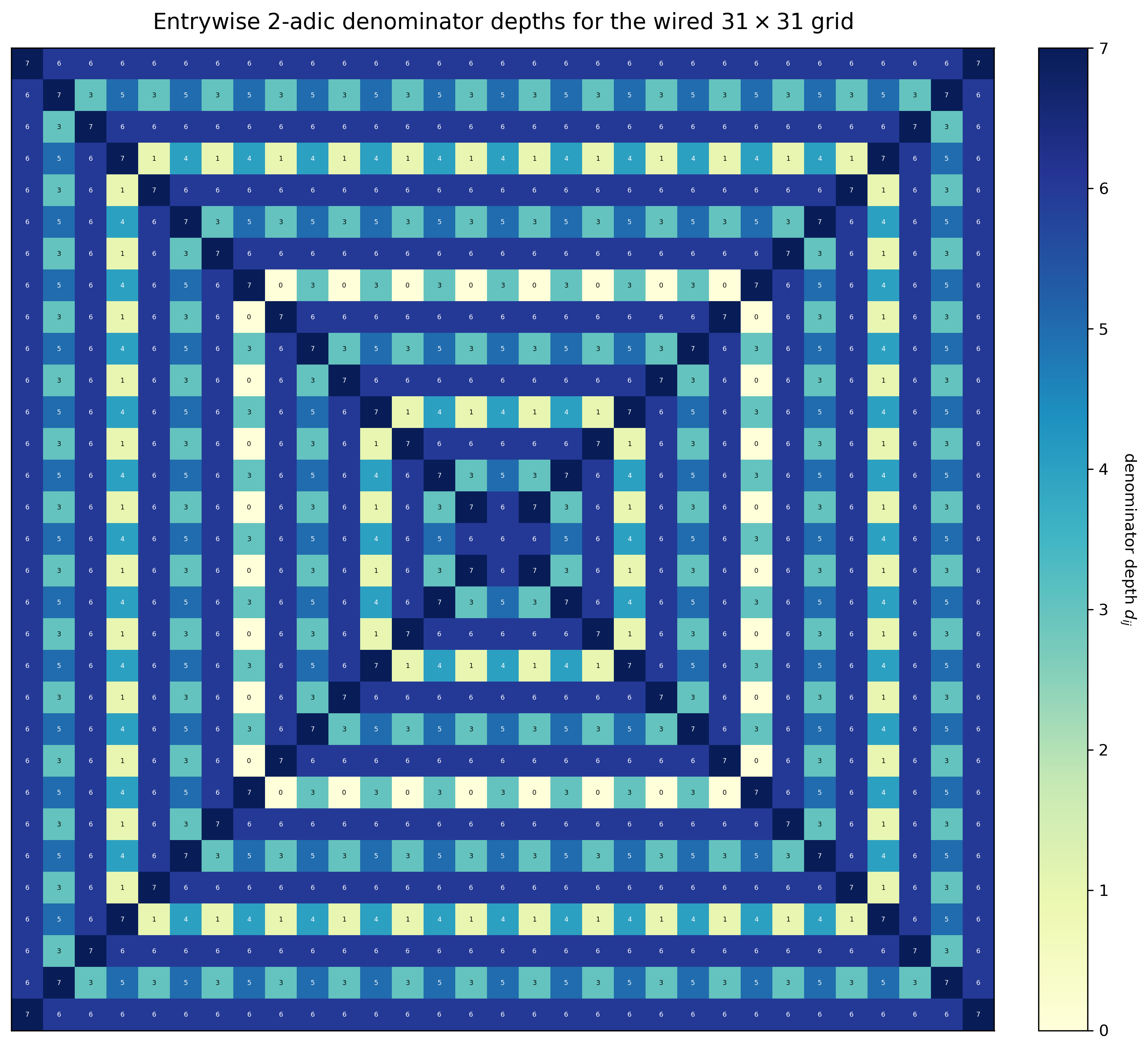}
\caption{The exact entrywise denominator depths $d^{(31)}_{ij}$ for the
wired $31\times31$ square.  Their maximum is $7$, in agreement with
the odd branch of Theorem~\ref{thm:main}, since
$\nu_2(31+1)+2=7$.  The nested pattern is not used in the proof, but
suggests that the scalar order $L(31)$ is the shadow of a finer filtered
structure.}
\label{fig:n31-depth}
\end{figure}

\section{Proof of the square-grid theorem}

We first extract the even-square case from the domino--sandpile
correspondence of Florescu, Morar, Perkinson, Salter, and Xu.  The odd branch
then has four exact reductions: reflection folds the grid to a quarter
square; a unimodular cyclic basis turns its one-dimensional operator into
multiplication by \(X\); taking the cokernel imposes \(Y=-X\); and two
normalized Euclidean remainders reduce modulo \(2\) to consecutive Fibonacci
polynomials.  The first three steps also justify the rectangular
constant-ideal reduction stated above.

\subsection{Even squares from the domino correspondence}

Let $N=2r$ with $r\ge1$.  By
\cite[Proposition~27 and Corollary~30]{FlorescuEtAl2015}, the order of the all-twos class
$[2\mathbf1]$ divides the odd integer $a_r$ in the square-board domino
count $2^r a_r^2$.  Since twice an element of order $L(N)$ has order
$L(N)/\gcd(L(N),2)$, it follows that $\nu_2(L(N))\le1$.

For the reverse inequality, reshape the integral vector
$z=L(N)A_N^{-1}\mathbf1$ as an $N\times N$ array.  Uniqueness of this
solution makes it invariant under all symmetries of the square.  At the
central vertex $(r,r)$, the right and lower neighbours therefore both have
value $z_{r,r}$, while the left and upper neighbours both have value
$z_{r-1,r}$.  For $r=1$, set $z_{0,1}=0$, the sink boundary value.  The
Laplacian equation at this vertex is
\[
L(N)=4z_{r,r}-2z_{r,r}-2z_{r-1,r}
=2(z_{r,r}-z_{r-1,r}).
\]
Thus $L(N)$ is even, proving
\begin{equation}\label{eq:even-square}
\ordtwo L(N)=1
\qquad(N\ge2\text{ even}).
\end{equation}
The upper bound is a consequence of their stated corollary; the central
parity observation supplies the exact valuation.

\subsection{Reflection reduction to a quarter square}

Write
\[
n=2h-1.
\]
Throughout the proof of the odd branch, $h=(n+1)/2\ge2$; the
case $h=1$ is treated separately at the end.
Let
\[
C_h=
\begin{pmatrix}
2&-1&&&\\
-1&2&-1&&\\
&\ddots&\ddots&\ddots&\\
&&-1&2&-1\\
&&&-2&2
\end{pmatrix}.
\]
Put
\[
Q_h=C_h\otimes I_h+I_h\otimes C_h,
\qquad e=(1,\ldots,1)^T\in\Z^h.
\]

\begin{lemma}\label{lem:reflection}
The solution $w=A_n^{-1}\mathbf 1$ is invariant under reflection in the horizontal and vertical symmetry axes. If $u$ is its restriction to the $h\times h$ quarter, then
\[
Q_hu=e\otimes e.
\]
For every integer $k$,
\[
kw\in\Z^{n^2}
\quad\Longleftrightarrow\quad
ku\in\Z^{h^2}.
\]
Consequently, the class of $\mathbf1$ in $\operatorname{coker}_{\Z}A_n$
and the class of $e\otimes e$ in $\operatorname{coker}_{\Z}Q_h$ have the
same order, namely $L(n)$.
\end{lemma}

\begin{proof}
Let
\[
J(x_1,\dots,x_h)^T
=(x_1,\dots,x_{h-1},x_h,x_{h-1},\dots,x_1)^T.
\]
Then $J$ is an isomorphism from $\Z^h$ onto the reflection-invariant sublattice of $\Z^n$, with inverse given by restriction to the first $h$ coordinates.  The matrix $A_n$ commutes with reflection in either coordinate, while $\mathbf1$ is fixed by both reflections.  Uniqueness of the solution of $A_nw=\mathbf1$ therefore shows that $w$ is invariant under both.  A direct calculation gives
\[
D_nJ=JC_h.
\]
Indeed, at the central coordinate the two reflected neighbours coincide, producing the coefficient $-2$ in the last row of $C_h$.

Tensoring gives
\[
A_n(J\otimes J)
=(J\otimes J)Q_h.
\]
Also $(J\otimes J)(e\otimes e)=\mathbf 1$.  Applying the intertwining
identity to the restriction $u$ of $w$ and using injectivity of
$J\otimes J$ gives $Q_hu=e\otimes e$ and
$w=(J\otimes J)u$.  Since restriction is a left inverse of $J\otimes J$,
$kw$ is integral exactly when $ku$ is integral.

If $Q_hz=0$, then the intertwining identity gives
$A_n(J\otimes J)z=0$.  Since $A_n$ and $J\otimes J$ are
injective, $z=0$.  Thus $Q_h$ is invertible over $\Q$, so the
quarter solution is unique.  Finally, for $k>0$, the relation
$k(e\otimes e)=Q_hy$ has the unique rational solution $y=ku$; similarly,
$k\mathbf1=A_nz$ has the unique rational solution $z=kw$.  The integrality
equivalence therefore proves equality of the two distinguished orders.
\end{proof}

\subsection{A unimodular cyclic basis and the Chebyshev quotient}

This is the point at which Chebyshev polynomials enter the general proof.
Their use in graph characteristic polynomials and spanning-tree
determinants is classical \cite{BoeschProdinger1986}.  What is needed here
in addition is an integral cyclic basis which remembers the all-ones
vector; a diagonalization over a field would not suffice for this purpose.
Define the first-kind polynomials by
\[
T_0(z)=1,\qquad T_1(z)=z,\qquad
T_{j+1}(z)=2zT_j(z)-T_{j-1}(z).
\]
The cosine addition formula gives
$T_j(\cos\theta)=\cos(j\theta)$.  Consequently, any identity between
Chebyshev expressions derived from the trigonometric formulas below holds
as a polynomial identity once it is verified at the infinitely many values
$z=\cos\theta\in[-1,1]$.

For the folded grid it is convenient to introduce the shifted sequence
\[
P_0(X)=2,\qquad P_1(X)=2-X,\qquad
P_{j+1}(X)=(2-X)P_j(X)-P_{j-1}(X).
\]
Comparing the two recurrences gives
\[
P_j(X)=2T_j\!\left(1-\frac X2\right).
\]
The recurrence shows directly that $P_j\in\Z[X]$ and that its leading
coefficient is $(-1)^j$.  For even
$h$ it is monic; for odd $h$ multiplying by $-1$ makes it monic.  This unit
will never matter for the ideals below.

\begin{lemma}\label{lem:cyclic}
The vectors
\[
e,\ C_he,\ C_h^2e,\ldots,C_h^{h-1}e
\]
form a $\Z$-basis of $\Z^h$.
\end{lemma}

\begin{proof}
One has
\[
C_he=e_1.
\]
For $1\le k\le h-1$, the vector $C_h^ke$ is supported in the first $k$ coordinates and its $k$th coordinate is $(-1)^{k-1}$. This follows inductively from tridiagonality: for $k\le h-2$, when one multiplies $C_h^ke$ by $C_h$, the only new coordinate is the $(k+1)$st, equal to minus the previous $k$th coordinate. Thus the Krylov matrix with columns
\[
e,C_he,\ldots,C_h^{h-1}e
\]
has determinant $\pm1$: expand along its last row, and the remaining
matrix, whose columns are $C_he,\ldots,C_h^{h-1}e$ restricted to the first
$h-1$ coordinates, is triangular with diagonal entries $\pm1$.
\end{proof}

\begin{lemma}\label{lem:charpoly}
The monic characteristic polynomial of $C_h$ is
\[
p_h(X):=\det(XI_h-C_h)
=2T_h\!\left(\frac{X-2}{2}\right)
=(-1)^hP_h(X).
\]
\end{lemma}

\begin{proof}
Continuant expansion from the first row, with the auxiliary value $p_0=2$,
gives
\[
p_0=2,\qquad p_1=X-2,\qquad
p_j=(X-2)p_{j-1}-p_{j-2}\quad(j\ge2).
\]
The polynomials $2T_j((X-2)/2)$ have the same initial values and recurrence,
so they equal $p_j$.  Finally, $T_j(-z)=(-1)^jT_j(z)$ gives
$p_j=(-1)^jP_j$.
\end{proof}

This polynomial is monic.  By Cayley--Hamilton, $p_h(C_h)=0$, and
Lemma~\ref{lem:cyclic} shows that
\begin{equation}\label{eq:integral-cyclic-map}
\Phi_h:\Z[X]/(p_h(X))\xrightarrow{\sim}\Z^h,
\qquad f(X)\longmapsto f(C_h)e
\end{equation}
is an integral isomorphism.  Indeed, the images of
$1,X,\ldots,X^{h-1}$ are the displayed cyclic basis.  Under $\Phi_h$,
multiplication by $X$ corresponds to $C_h$, while $1$ corresponds to $e$.
The determinant $\pm1$ in Lemma~\ref{lem:cyclic} is essential: it makes
this an isomorphism of lattices, not merely of rational vector spaces, so
no denominator information is lost.

\begin{proposition}[Direct cyclic-cokernel reduction]
\label{prop:cyclic-cokernel}
There is an isomorphism of abelian groups
\begin{equation}\label{eq:integral-cokernel}
\operatorname{coker}_{\Z}Q_h
\xrightarrow{\sim}
\Z[X]/(P_h(X),P_h(-X))
\end{equation}
which sends the class of $e\otimes e$ to the class of $1$.
\end{proposition}

\begin{proof}
Tensoring two copies of \eqref{eq:integral-cyclic-map} identifies
$\Z^h\otimes\Z^h$ with
\[
\Z[X,Y]/(p_h(X),p_h(Y)).
\]
In these coordinates $Q_h=C_h\otimes I_h+I_h\otimes C_h$ is multiplication
by $X+Y$, and $e\otimes e$ is $1$.  Taking the cokernel therefore gives
\[
\operatorname{coker}_{\Z}Q_h
\simeq
\frac{\Z[X,Y]}{(p_h(X),p_h(Y),X+Y)}
\simeq
\frac{\Z[X]}{(p_h(X),p_h(-X))},
\]
where the last map substitutes $Y=-X$.  Multiplying either generator by
the unit $(-1)^h$ replaces $p_h$ by $P_h$ without changing the ideal.
The distinguished class is $1$ throughout.
\end{proof}

\subsection{The 2-adic constant ideal}

The direct cokernel description makes the local order immediate.  The group
$\operatorname{coker}_{\Z}Q_h$ is finite because $Q_h$ is nonsingular.
For any finite abelian group, tensoring with $\Ztwo$ retains precisely its
$2$-primary component.  Tensoring the presentation in
Proposition~\ref{prop:cyclic-cokernel} therefore gives
\[
(\operatorname{coker}_{\Z}Q_h)\otimes_{\Z}\Ztwo
\simeq
\Ztwo[X]/(P_h(X),P_h(-X)),
\]
and the distinguished class still corresponds to $1$.  Lemma~\ref{lem:reflection}
then yields the exact reduction
\begin{equation}\label{eq:idealreduction}
\boxed{
\ordtwo L(n)
=
\min\{r\ge0:2^r\in(P_h(X),P_h(-X))\subset\Ztwo[X]\}.
}
\end{equation}
Equivalently, the problem is to determine the constant ideal
\[
(P_h(X),P_h(-X))\cap\Ztwo.
\]
Because $p_h=(-1)^hP_h$, this is also
\[
Y(p_h,p_h(-X)),
\qquad
Y(f,g):=(f,g)\cap\Ztwo.
\]

We need only the following two elementary Euclidean moves.

\begin{lemma}[Scaling and stopping]\label{lem:ideal-moves}
Let $f,g\in\Ztwo[X]$.
\begin{enumerate}[label=\textnormal{(\roman*)}]
\item Adding a polynomial multiple of one generator to the other, swapping
the generators, or multiplying a generator by a scalar unit does not change
$Y(f,g)$.
\item If $\deg f\ge1$ and the leading coefficient of $f$ is a unit, then
\[
Y(f,2^kg)=2^kY(f,g)\qquad(k\ge0).
\]
\item If both leading coefficients are units and $\bar f,\bar g$ (the reduction modulo $2$ of $f$ and $g$) are
coprime in $\mathbb F_2[X]$, then $Y(f,g)=\Ztwo$.
\end{enumerate}
\end{lemma}

\begin{proof}
Part (i) is the usual row operation on a two-generated ideal.  For (ii),
normalize $f$ to be monic and work in the free $\Ztwo$-module
$A=\Ztwo[X]/(f)$.  If a scalar $c$ lies in the image of $2^kg$, then
$c\in2^kA$; comparison in the basis $1,X,\ldots,X^{\deg f-1}$ gives
$c=2^kc_0$ with $c_0\in\Ztwo$.  Since $A$ has no $2$-torsion, cancellation
of $2^k$ shows that $c_0$ lies in the image of $g$.  The reverse inclusion
is immediate after multiplying a relation by $2^k$.

For (iii), if either polynomial is constant it is already a unit.
Otherwise normalize both polynomials to be monic.  Their resultant belongs
to $(f,g)\cap\Ztwo$ and reduces to the nonzero resultant of
$\bar f,\bar g$.  It is therefore an odd scalar, hence a unit; so the
constant ideal is all of $\Ztwo$.
\end{proof}

\begin{remark}[Why resultants appear here]
Strehl's resultant treatment of rectangular domino counts
\cite{Strehl2001} and the explicit Chebyshev resultant formulas of
Jacobs, Rayes, and Trevisan \cite{JacobsRayesTrevisan2011} place this
calculation in a familiar determinant framework.  The intersection with
the constants asks a finer question.  The resultant of $P_h(X)$ and
$P_h(-X)$ belongs to their constant ideal, but need not generate it: its
valuation is the exponent of $2$ in the cardinality of the whole finite
quotient, whereas the generator of the constant ideal records the order
of the class $1$.
Thus a resultant supplies an annihilating constant, not necessarily the
smallest possible power of $2$.
\end{remark}

\subsection{Two Euclidean remainders and Fibonacci polynomials}

For a polynomial over \(\Ztwo\), write \(\bar f\) for its reduction modulo
\(2\).  Define the Fibonacci polynomials in \(\mathbb F_2[X]\) by
\[
F_0=0,\qquad F_1=1,\qquad F_j=XF_{j-1}+F_{j-2}\quad(j\ge2).
\]
Induction gives
\begin{equation}\label{eq:fibonacci-coefficients}
F_j(X)=
\sum_{\ell=0}^{\lfloor(j-1)/2\rfloor}
\binom{j-\ell-1}{\ell}X^{j-1-2\ell}.
\end{equation}
The recurrence also gives
\[
(F_j,F_{j-1})=(F_{j-1},F_{j-2})=\cdots=(F_1,F_0)
=\mathbb F_2[X].
\]
Thus consecutive Fibonacci polynomials are coprime.

Fix \(h\ge2\) and put
\[
b=\ordtwo(h),\qquad h=2^bq,\quad q\in\Ztwo^\times,
\qquad s=(-1)^h.
\]
In \(\Qtwo[X]\), define
\begin{equation}\label{eq:two-remainders}
\begin{aligned}
R_h(X)&:=\frac{p_h(-X)-s p_h(X)}{2^{b+2}},\\
\Lambda_h(X)&:=sq^{-1}(X-2h),\\
H_h(X)&:=\frac{p_h(X)-\Lambda_h(X)R_h(X)}2.
\end{aligned}
\end{equation}

\begin{lemma}[The remainder calculation]\label{lem:fibonacci-remainders}
The polynomials \(R_h,H_h\) belong to \(\Ztwo[X]\), have respective
degrees \(h-1,h-2\), and have unit leading coefficients.  Moreover,
\[
\overline{R}_h=F_h,
\qquad
\overline{H}_h=F_{h-1}.
\]
\end{lemma}

\begin{proof}
The continuant recurrence for \(p_h\) gives, by induction,
\begin{equation}\label{eq:ph-explicit}
p_h(X)=
\sum_{k=0}^{h}
(-1)^{h+k}\frac{2h}{h+k}\binom{h+k}{2k}X^k.
\end{equation}
In the numerator defining \(R_h\), the coefficient of \(X^k\) vanishes
when \(k\equiv h\pmod2\).  For
\(0\le j\le\lfloor(h-1)/2\rfloor\), direct substitution in
\eqref{eq:ph-explicit} gives
\begin{equation}\label{eq:Rh-coefficients}
[X^{h-1-2j}]R_h
=
\frac{sq}{2h-1-2j}
\binom{2h-1-2j}{2j+1}.
\end{equation}
The denominator is odd, so these coefficients lie in \(\Ztwo\); at
\(j=0\) the coefficient is \(sq\), a unit.  Over \(\mathbb F_2\),
\[
\binom{2a+1}{2c+1}\equiv\binom ac\pmod2,
\]
as follows by comparing coefficients in
\((1+Z)^{2a+1}=(1+Z)(1+Z^2)^a\).  Formula
\eqref{eq:Rh-coefficients} therefore reduces to
\(\binom{h-j-1}{j}\), and \eqref{eq:fibonacci-coefficients} gives
\(\bar R_h=F_h\).

The top two coefficients of \(p_h\) are \(1\) and \(-2h\).
The polynomial \(R_h\) has leading coefficient \(sq\) and contains only
degrees of parity opposite to \(h\).  Hence ordinary division of \(p_h\)
by \(R_h\) has quotient \(\Lambda_h=sq^{-1}(X-2h)\).  More explicitly,
if \(p_h^{\mathrm{opp}}\) denotes the part of \(p_h\) in degrees of parity
opposite to \(h\), then
\[
p_h^{\mathrm{opp}}=-s2^{b+1}R_h,
\]
which is exactly the product of \(R_h\) with the constant term of
\(\Lambda_h\).  Thus all opposite-parity terms cancel in
\(p_h-\Lambda_hR_h\), as do the terms of degrees \(h\) and \(h-1\).

It remains to inspect degrees \(h-2j\), where
\(1\le j\le\lfloor h/2\rfloor\).  Put \(M=h-j\).  Substitution of
\eqref{eq:ph-explicit} and \eqref{eq:Rh-coefficients}, followed by
cancellation, gives
\begin{equation}\label{eq:Hh-coefficients}
[X^{h-2j}]H_h
=\frac12\left(
\frac hM\binom{2M}{2j}
-\frac1{2j+1}\binom{2M-2}{2j}
\right)
=
\binom{M-1}{j-1}u_{M,j},
\end{equation}
where
\[
u_{M,j}
=
\frac{2M^2+2Mj+2M-2j-1}{(2M-1)(2j+1)}
\frac{\binom{2M}{2j}}{\binom Mj}.
\]
The first fraction is a quotient of odd integers.  Also
\[
\ordtwo((2r)!)=r+\ordtwo(r!)
\]
implies
\(\ordtwo\binom{2M}{2j}=\ordtwo\binom Mj\).
Hence \(u_{M,j}\in\Ztwo^\times\), and so \(u_{M,j}\equiv1\pmod2\).
Formula \eqref{eq:Hh-coefficients} now reduces to
\(\binom{h-j-1}{j-1}\).  Relabeling \(\ell=j-1\) in
\eqref{eq:fibonacci-coefficients} proves
\(\bar H_h=F_{h-1}\).  Its leading coefficient is therefore a unit and
its degree is \(h-2\).
\end{proof}

\begin{proof}[Completion of the odd branch of Theorem~\ref{thm:main}]
The definitions in \eqref{eq:two-remainders} are the two exact Euclidean
relations
\[
p_h(-X)=s p_h(X)+2^{b+2}R_h(X),
\qquad
p_h(X)=\Lambda_h(X)R_h(X)+2H_h(X).
\]
Using Lemma~\ref{lem:ideal-moves} twice and then
Lemma~\ref{lem:fibonacci-remainders}, we obtain
\begin{align*}
Y(p_h,p_h(-X))
&=2^{b+2}Y(p_h,R_h)\\
&=2^{b+3}Y(R_h,H_h)
 =2^{b+3}\Ztwo,
\end{align*}
because \(\bar R_h=F_h\) and \(\bar H_h=F_{h-1}\) are coprime.
Equation~\eqref{eq:idealreduction} therefore gives
\[
\ordtwo L(n)=b+3
=\ordtwo(2h)+2
=\ordtwo(n+1)+2.
\]
This proves the theorem for \(h\ge2\), equivalently for odd \(n\ge3\).
For \(h=1\), the quotient by \(p_1(X)=X-2\) evaluates
\(p_1(-X)\) at \(X=2\), giving \(-4\); hence \(L(1)=4\), as stated.
Together with \eqref{eq:even-square}, this completes the proof of
Theorem~\ref{thm:main}.
\end{proof}

\begin{remark}[Why Fibonacci polynomials appear]
Modulo \(2\), the shifted Chebyshev recurrence becomes the Fibonacci
recurrence.  The involution \(X\mapsto-X\) first removes the terms of one
parity; after the exact powers of \(2\) are divided out, the two surviving
remainders are precisely \(F_h\) and \(F_{h-1}\).  Their coprimality is the
reason the valuation is exhausted after two steps.
\end{remark}

\section{Further directions}

\paragraph{The full group and its distinguished class.}
The proof identifies one distinguished element of a sandpile cokernel.
Can one determine the full $2$-primary Smith form of the same grids and
locate the all-ones class inside that decomposition?  Kuperberg's integral
approach to matching matrices is a natural precedent
\cite{Kuperberg2002}.  Spectral factorizations alone need not answer this
question: equality between valuations of eigenvalues and local Smith
invariants requires additional hypotheses \cite{ElsheikhGiesbrecht2015},
and even a Smith form must be accompanied by the coordinates of the chosen
class.  Maximum element orders in other graph families, such as strongly
regular graphs, provide a comparison \cite{HungYuen2022}, but the exponent
of a group need not be the order of its all-ones element.  Could a
monodromy-pairing calculation \cite{Shokrieh2010} detect the latter
directly, or explain the finer entrywise pattern in
Figure~\ref{fig:n31-depth}?

\paragraph{Interpolation and graph towers.}
Cohn's theorem suggests asking whether the odd part
$L(n)/2^{\nu_2(L(n))}$ has a continuous or analytic interpolation on
suitable residue classes in $\Ztwo$.  Graph analogues of Iwasawa theory
describe $p$-primary growth of entire graph Jacobians in voltage
$p$-towers and $\Z_\ell^d$-towers
\cite{Gonet2022,DuBoseVallieres2023}.  The wired squares here are not
supplied with such a tower structure, so those results do not provide
this interpolation.  Is there a compatible family of graphs and maps
which retains a canonical class analogous to the all-ones class, and for
which its order can be studied alongside the order of the whole group?

\paragraph{A matching interpretation of the period.}
Does the denominator of $A_n^{-1}\mathbf1$ have a direct matching or
Pfaffian interpretation extending the even-square argument?  Beyond the
determinant correspondence, Kuperberg studies the associated cokernels
\cite{Kuperberg2002}.  For the planar bipartite graphs obtained by the
Kenyon--Propp--Wilson construction, Taylor constructs a simply transitive
action of the Kasteleyn cokernel on perfect matchings
\cite{KPW2000,Taylor2023}.  In that setting, a specified group element has
a concrete action on matchings.  Can the all-ones class be transported to
such an action so that $L(n)$ becomes an orbit length?  Taylor's theorem
is not a statement about every planar bipartite graph, and such a
transport must keep track of the distinguished element.

\paragraph{Reflections and the exceptional prime.}
For genuine graph covers, Reiner and Tseng relate the critical groups of
the cover, the base, and an associated voltage graph; their exact sequence
splits on primary components away from primes dividing the covering
degree \cite{ReinerTseng2014}.  This makes the prime $2$ a natural place
to look for additional structure under an involution.  Our odd-grid
reflections have fixed vertices, however, so the fold is not a regular
double cover and their splitting theorem does not apply directly.  Can
an analogous integral description with fixed vertices explain both the
folded all-ones class and the remaining $2$-primary factors?

\paragraph{Chebyshev arithmetic.}
Chebyshev arithmetic suggests additional number-theoretic questions.  For
example, $T_r$ is irreducible over $\mathbb Q$ exactly when $r$ is a power
of $2$, and Chebyshev maps have rigid dynamical decompositions over $\Ztwo$
\cite{RTW2005,FanLiao2016}.  These facts are not needed by the shortest
proof below, but they show that its dyadic behavior is part of a wider
arithmetic pattern.  For every odd prime $p$ one has the Fermat-type congruence
$T_p(x)\equiv x^p\pmod p$ \cite{Bang1954}, and the composite indices
satisfying the analogous universal congruence have a Korselt-style
characterization \cite{JacobsRayesTrevisan2008}.  Special-value sequences
also exhibit Zsigmondy-type primitive divisors \cite{Baranczuk2019}.  

What makes the argument distinctive is the coexistence of four structures:
reflection acts integrally on the grid, a cyclic vector exposes a one-variable
operator, the involution $X\mapsto-X$ exposes the exact powers of $2$, and
consecutive Fibonacci polynomials stop the normalized Euclidean algorithm.
The even-square valuation comes from the domino correspondence and a parity
obstruction; the odd-square formula is the visible trace of all four
structures above.

\section*{Acknowledgments}
After numerous experiments leading to the right conjectures, the initial proof was obtained by ChatGPT, and then simplified by the authors.

Higinio Serrano, Turgay Akyar, Mikhail Shkolnikov and Ernesto Lupercio are supported by the Simons Foundation, grant SFI-MPS-T-Institutes-00007697, and the Ministry of Education and Science of the Republic of Bulgaria, grant DO1-239/10.12.2024.

\end{document}